\documentclass[11pt,en,oneside,bbfont,article, reqno]{amsart} 
\usepackage[utf8]{inputenc}
\usepackage{datetime}
\usepackage[a4paper,margin=1in]{geometry}
\usepackage{amsfonts}
\usepackage{amsmath}
\usepackage{amssymb}
\usepackage{nicefrac}

\usepackage{amsthm}
\usepackage{csquotes}
\usepackage{tikz}
\usepackage{textcomp}
\usepackage{arydshln}
\usepackage{float}
\usepackage[dvipsnames]{xcolor}

\usepackage{breqn}
\usepackage{thmtools}
\usepackage{multirow}
\usepackage{mathtools}\mathtoolsset{showonlyrefs}
\usepackage{graphicx}
\usepackage{subcaption}

\usepackage{booktabs}
\usepackage{listings}

\usepackage{letltxmacro}

\LetLtxMacro\OriginalLongrightarrow\Longrightarrow
\LetLtxMacro\OriginalLongleftarrow\Longleftarrow

\usepackage{trimclip}
\DeclareRobustCommand\Longrightarrow{\NewRelbar\joinrel\Rightarrow}
\DeclareRobustCommand\Longleftarrow{\Leftarrow\joinrel\NewRelbar}

\makeatletter
\DeclareRobustCommand\NewRelbar{
	\mathrel{
		\mathpalette\@NewRelbar{}
	}
}
\newcommand*\@NewRelbar[2]{
	\sbox0{$#1=$}
	\sbox2{$#1\Rightarrow\m@th$}
	\sbox4{$#1\Leftarrow\m@th$}
	\clipbox{0pt 0pt \dimexpr(\wd2-.6\wd0) 0pt}{\copy2}
	\kern-.2\wd0
	\clipbox{\dimexpr(\wd4-.6\wd0) 0pt 0pt 0pt}{\copy4}
}
\makeatother

\definecolor{seagreen}{HTML}{21b2aa}
\definecolor{magenta}{HTML}{b2217f}
\definecolor{gold}{HTML}{ca9520}
\definecolor{red}{HTML}{da272f}
\definecolor{blue}{HTML}{4682b4}
\definecolor{navy}{HTML}{7b000b}
\definecolor{nred}{HTML}{650015}
\definecolor{ngreen}{HTML}{156515}
\definecolor{npurple}{HTML}{450075}
\definecolor{ngreenalt}{HTML}{278e40}

\definecolor{cream-white}{HTML}{f5f5f0}
\definecolor{off-black}{HTML}{1a1a1a}

\definecolor{highlightcolor}{named}{navy}

\definecolor{codebg}{HTML}{f0f0f0}
\definecolor{commentfg}{HTML}{333333}

\lstdefinestyle{mystyle}{
	backgroundcolor=\color{codebg},
	commentstyle=\color{commentfg},
	keywordstyle=\color{navy},
	stringstyle=\color{commentfg},
	basicstyle=\ttfamily\footnotesize,
	breakatwhitespace=false,
	breaklines=true,
	captionpos=b,
	keepspaces=true,
	numbers=none,
	showspaces=false,
	showstringspaces=false,
	showtabs=false,
	tabsize=2
}
\usepackage{enumitem}

\usepackage[
	style=numeric,
	giveninits,
	sortcites=true
]{biblatex}
\let\oldcite\cite
\renewcommand{\cite}[2][]{%
	\def\tmp{#1}%
	\ifx\tmp\empty
		\mbox{\oldcite{#2}}%
	\else
		\mbox{\oldcite[\tmp]{#2}}%
	\fi
}
\AtEveryBibitem{\clearfield{doi}}
\AtEveryBibitem{\clearfield{url}}
\AtEveryBibitem{\clearfield{isbn}}
\AtEveryBibitem{\clearfield{issn}}

\usepackage{pgfplots}
\pgfplotsset{compat=1.15}
\usepackage{mathrsfs}
\usetikzlibrary{arrows}

\usepackage[section]{placeins}

\newcommand{\ub}{\rightarrow \infty}

\newcommand{\st}{\, \colon \,}
\newcommand{\D}{\,\mathrm d}

\newcommand{\jpod}{\stackrel{\mathrm d}{=}}

\newcommand{\konvd}{\stackrel{\mathrm d}{\longrightarrow}}

\DeclareMathOperator{\ctg}{cot}

\DeclareMathOperator{\per}{perim}
\DeclareMathOperator{\area}{area}
\DeclareMathOperator{\hull}{hull}
\DeclareMathOperator*{\argmax}{argmax}

\newcommand{\ol}[1]{\overline{#1}}

\newcommand{\R}{\mathbb{R}}

\renewcommand{\P}{\mathbb{P}}

\newcommand{\E}{\mathbb{E}}

\DeclarePairedDelimiter{\vit}{\{}{\}}
\let\oldvit\vit
\renewcommand{\vit}[1]{\oldvit*{#1}}

\DeclarePairedDelimiter{\ug}{[}{]}
\let\oldug\ug
\renewcommand{\ug}[1]{\oldug*{#1}}

\DeclarePairedDelimiter{\ob}{(}{)}
\let\oldob\ob
\renewcommand{\ob}[1]{\oldob*{#1}}

\DeclarePairedDelimiterX{\bob}[1]{\Big(}{\Big)}{#1}
\DeclarePairedDelimiterX{\bug}[1]{\Big[}{\Big]}{#1}
\DeclarePairedDelimiterX{\bvit}[1]{\Big\{}{\Big\}}{#1}

\DeclarePairedDelimiter{\skp}{\langle}{\rangle}

\renewcommand{\skp}[2]{\oldskp*{#1,#2}}
\renewcommand{\skp}[2]{ {#1} \cdot {#2} }

\DeclarePairedDelimiter{\oo}{(}{)}
\let\oldoo\oo
\renewcommand{\oo}[1]{\oldoo*{#1}}

\DeclarePairedDelimiter{\oc}{(}{]}
\let\oldoc\oc
\renewcommand{\oc}[1]{\oldoc*{#1}}

\DeclarePairedDelimiter{\co}{[}{)}
\let\oldco\co
\renewcommand{\co}[1]{\oldco*{#1}}

\DeclarePairedDelimiter{\cc}{[}{]}
\let\oldcc\cc
\renewcommand{\cc}[1]{\oldcc*{#1}}

\newcommand{\up}[1]{\mathrm{#1}}

\renewcommand{\le}{\leqslant}
\renewcommand{\ge}{\geqslant}

\newcommand{\nderi}[1]{^{(#1)}}

\newcommand{\restr}[2]{{
\left.\kern-\nulldelimiterspace
#1
\vphantom{\big|}
\right|_{#2}
}}

\makeatletter
\DeclareRobustCommand{\lasymp}{\lg@asymp{<}}
\DeclareRobustCommand{\gasymp}{\lg@asymp{>}}

\newcommand{\under@asymp}[1]{\clipbox{0pt 0pt 0pt {0.5\height}}{$\m@th#1\asymp$}}
\newcommand{\lg@asymp}[1]{\mathrel{\mathpalette\lg@asymp@{#1}}}
\newcommand{\lg@asymp@}[2]{%
	\vcenter{%
		\offinterlineskip
		\m@th
		\ialign{%
			\hfil##\hfil\cr
			$#1#2$\cr
			\under@asymp{#1}\cr
		}%
	}%
}
\makeatother

\newcommand{\astfootnote}[1]{%
	\let\oldthefootnote=\thefootnote%
	\setcounter{footnote}{0}%
	\renewcommand{\thefootnote}{\fnsymbol{footnote}}%
	\footnote{#1}%
	\let\thefootnote=\oldthefootnote%
}

\usepackage[unicode]{hyperref}
\hypersetup{
	colorlinks,
	linkcolor=highlightcolor,
	citecolor=highlightcolor,
	urlcolor=highlightcolor,
	hypertexnames=false
}

\newtheorem{theorem}{Theorem}[section]
\newtheorem{proposition}[theorem]{Proposition}
\newtheorem{lemma}[theorem]{Lemma}
\theoremstyle{remark}
\newtheorem{remark}[theorem]{Remark}
\newtheorem{corollary}[theorem]{Corollary}

\newcommand{\iproc}[2]{X_{#1}(#2)}
\newcommand{\iprocs}[1]{X_{#1}}
\newcommand{\gproc}[3]{{#1}_{#2}(#3)}
\newcommand{\gprocs}[2]{{#1}_{#2}}
\newcommand{\si}{\sigma}
\renewcommand{\top}{\tau}

\title[On the convex hull of a Gaussian-endpoint Brownian bridge]{On the convex hull of a planar Brownian bridge with a random Gaussian endpoint}
\author{Nikola Sandri\'{c}$^*$, Stjepan \v{S}ebek$^{**}$ \lowercase{and} Luka \v{S}imek$^*$}

\address{$^*$Department of Mathematics, University of Zagreb, HR}
\email{nsandric@math.hr}
\email{lsimek@math.hr}

\address{$^{**}$University of Zagreb Faculty of Electrical Engineering and Computing, HR}
\email{stjepan.sebek@fer.unizg.hr}

\begin{document}

\begin{abstract}
	We consider a one-parameter family of isotropic planar Gaussian processes
	\[
		\iproc \si t =B_t+\sigma t Z,\qquad 0\le t\le 1,\quad 0\le \sigma\le 1,
	\]
	where $B$ is a standard ($0$-to-$0$) planar Brownian bridge on $[0,1]$, and $Z\sim \up N(0,I)$ is a standard Gaussian random vector independent of $B$. The family interpolates between standard planar Brownian bridge ($\sigma=0$) and standard planar Brownian motion ($\sigma=1$). As the main result of the paper we compute the expected perimeter and area of the convex hull of the random set $\vit{\iproc \si t \st 0\le t\le 1}$ as closed formulas in terms of $\sigma$, and recover the classical Brownian bridge and Brownian motion values at $\sigma=0$ and $\sigma=1$. We also consider the convex hull spanned by multiple independent processes of this type and the possibilities for closed formulas in special cases. The key observation in our argument is that the isotropy property reduces the expected perimeter and area to one-dimensional quantities through the support function and Cauchy's formulas.
\end{abstract}


\maketitle

\section{Introduction}
Let $B=\vit{B_t \st 0 \le t \le 1}$ be a standard ($0$-to-$0$) planar Brownian bridge and let
$Z\sim \up N(0,I)$ be a standard planar Gaussian random vector independent of $B$. For a parameter $\sigma\in[0,1]$, define the \emph{Gaussian-endpoint Brownian bridge}
\begin{equation}\label{eq:def-process}
	\iproc \si t=  B_t+\sigma t Z,\qquad 0\le t\le 1.
\end{equation}

\begin{figure}[htbp]
	\centering

	\begin{subfigure}[t]{0.48\textwidth}
		\centering
		\includegraphics[width=\linewidth]{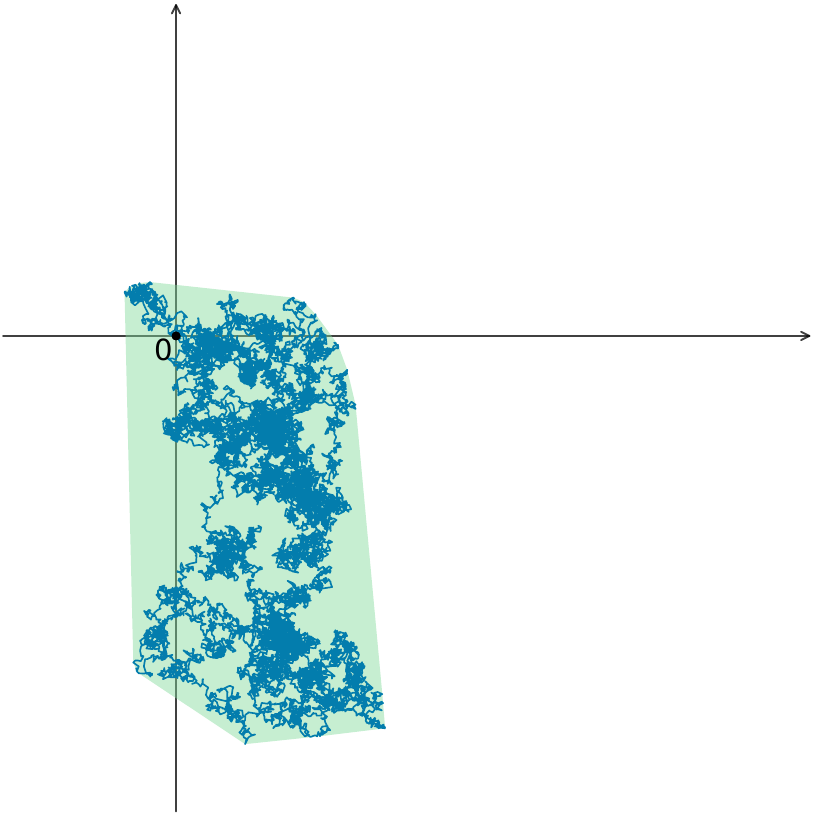}
		\caption{$\sigma=0$: Brownian bridge.}
		\label{fig:bb-random-endpoint-sigma-0}
	\end{subfigure}
	\hfill
	\begin{subfigure}[t]{0.48\textwidth}
		\centering
		\includegraphics[width=\linewidth]{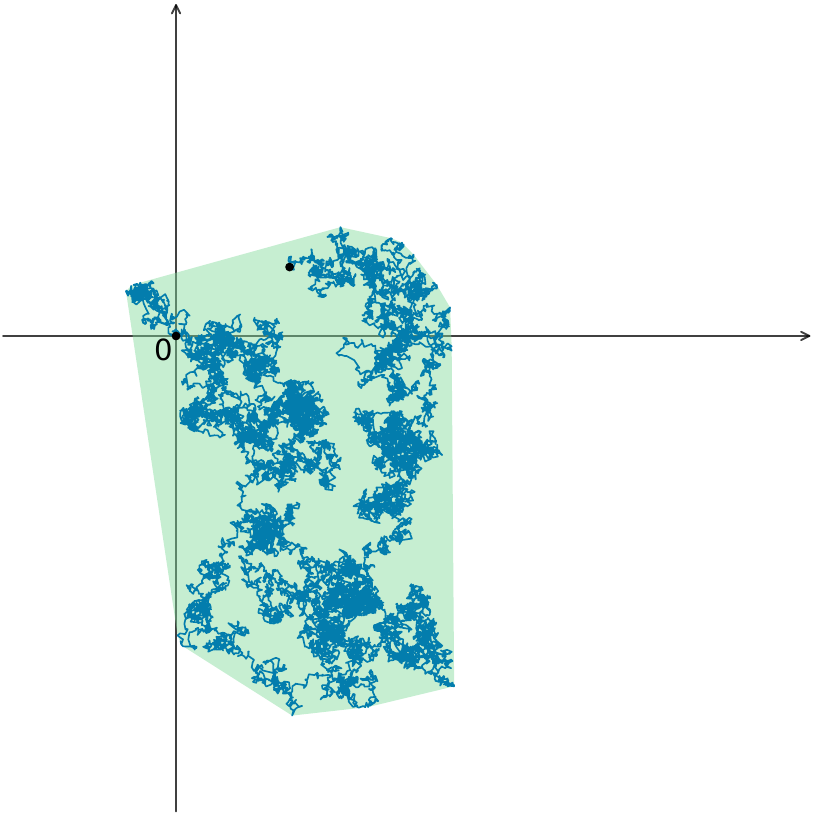}
		\caption{$\sigma=0.3$.}
		\label{fig:bb-random-endpoint-sigma-03}
	\end{subfigure}

	\vspace{0.8em}

	\begin{subfigure}[t]{0.48\textwidth}
		\centering
		\includegraphics[width=\linewidth]{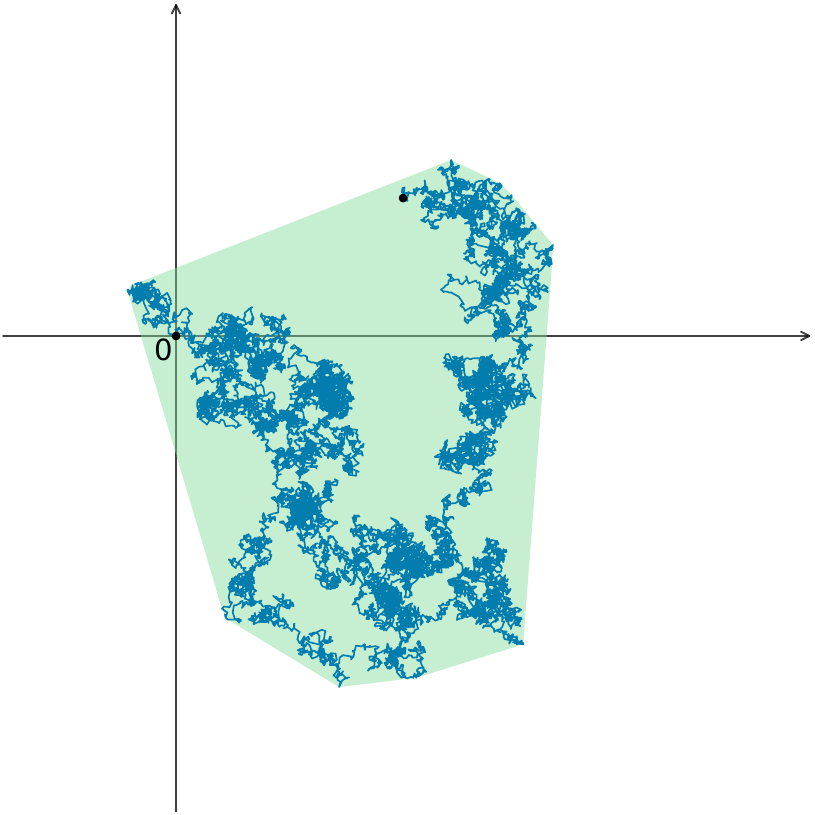}
		\caption{$\sigma=0.6$.}
		\label{fig:bb-random-endpoint-sigma-06}
	\end{subfigure}
	\hfill
	\begin{subfigure}[t]{0.48\textwidth}
		\centering
		\includegraphics[width=\linewidth]{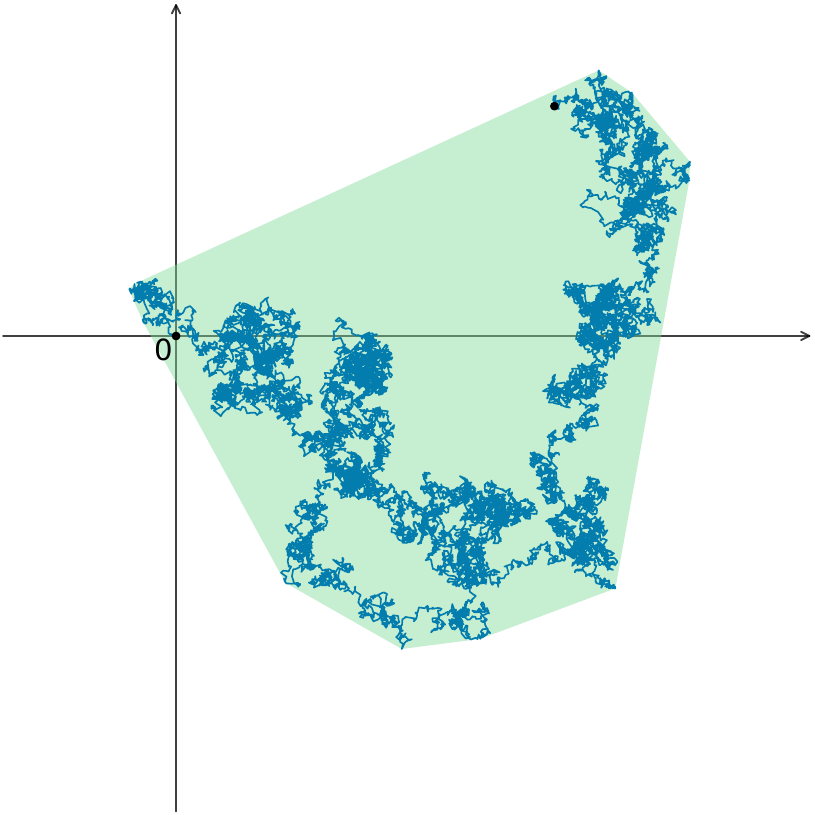}
		\caption{$\sigma=1$: Brownian motion.}
		\label{fig:bb-random-endpoint-sigma-1}
	\end{subfigure}

	\caption{The same planar Brownian bridge realization transformed into Gaussian-endpoint bridges, together with their convex hulls.}
	\label{fig:1}
\end{figure}

The process \( \iprocs \si \) is a standard planar Brownian bridge when \( \si=0 \) and planar Brownian motion when
\( \si=1 \) (see Figure~\ref{fig:1}). Both claims are clear from its covariance structure derived in Subsection~\ref{ssec:prelim-basics}.
Conditionally on \(Z\), the process in~\eqref{eq:def-process}
is a planar Brownian bridge from \(0\) to the endpoint \(\sigma Z\). One-dimensional Brownian bridges with random or unknown pinning points have appeared in several contexts, especially in optimal stopping and financial modeling. For instance, in~\cite{leung2018optimal} the authors studied randomized Brownian bridges with various prior distributions for the random endpoint, motivated by optimal trading problems. In~\cite{ekstrom_vaicenavicius} the optimal stopping of a Brownian bridge with an unknown pinning point is considered, and in~\cite{louriki} the author studies Brownian bridges with random length and random pinning point in financial information models.

In this paper, we focus on the convex hull spanned by the process' trajectory
\begin{equation}
	K_\si = \hull \vit{\iproc \si t \st 0 \le t \le 1}.
\end{equation}
The convex hull of planar Brownian motion has been studied for a long time. Already L\'{e}vy in~\cite{Levy} considered certain properties of its boundary and conjectured that the boundary of the Brownian convex hull $K_1$ should be of class $C^1$. This conjecture was proved in~\cite{elbachir}; see also~\cite{Cranston} for another proof. Our focus in this paper is on certain geometric functionals of these random sets, namely the perimeter and area. The first paper dealing with the perimeter of $K_1$ is by Letac and Takacs~\cite{letac_takacs} where it was shown that $\E [\per K_1]  = \sqrt{8\pi}$. The result for Brownian bridge is due to Goldman~\cite{goldman} who showed that $\E [\per K_0] = \sqrt{\pi^3/2}$. When it comes to the area, El Bachir in~\cite{elbachir} obtained that $\E [\area K_1] = \pi/2$, and this result was generalized to the Brownian bridge by Majumdar et al.\ in~\cite{majumdar} where it was proved that $\E[ \area K_0] = \pi/3$. In the same paper, the above results have been generalized to convex hulls of multiple independent standard planar Brownian motions, and to convex hulls of multiple independent standard planar Brownian bridges. In~\cite{sebek} the convex hull of a combination of Brownian motions and bridges was considered.
Other work in this field deals with random walks~\cite{Spitzer-Widom, Baxter, Snyder-Steele, Vysotsky-Zapor, Cygan-Sandric-Sebek, CSSW, Akopyan_Vysotsky, KVZ, LH_W, McR_W, Wade_Xu-PAMS, Wade_Xu-SPA}, processes in multiple dimensions~\cite{Kabluchko-Zapor-TAMS, Eldan, Molchanov-Wespi, Molchanov, high_dim_hulls}, and other functionals such as the diameter~\cite{McRedmond_Xu, Jovalekic, CPS}.

This paper deals with the perimeter and area of the convex hull $K_\si$,
as well as the convex hull spanned by multiple independent processes of the type. Thereby
we also generalize a number of the above results. The guiding principle is the same as in the classical support function approach to planar random convex hulls---by isotropy, one can reduce the Cauchy formulas for perimeter and area to a single direction, hence to a one-dimensional projected process.
This is justified in Section~\ref{sec:preliminaries}.
In Section~\ref{sec:perimter} we study the expected perimeter of  $K_\si$. Our main result states that
\begin{equation*}
	\E [\per K_\si] = \sqrt{2\pi} \ob{\sigma+\frac{\arccos\sigma}{\sqrt{1-\sigma^2}}}.
\end{equation*}
Section~\ref{sec:area} deals with the expected area of  $K_\si$. We prove that
\begin{equation}
	\E [\area K_\si] = \frac \pi 3 \cdot \frac{\si^2 + \si + 1}{\si + 1}.
\end{equation}
In Section~\ref{sec:multiple_processes}, the convex hull spanned by multiple independent processes is studied. We obtain a closed formula for the expected perimeter of the convex hull spanned by two independent processes (see Theorem~\ref{tm:Eperim2proc}). We argue that a closed formula seems not to exist for the expected area or whenever more than two independent processes are considered. Finally, we briefly consider the asymptotics of the shape, as well as the expected perimeter and area, as the number of processes tends to infinity. These are corollaries of results shown by Davydov~\cite{davydov}.

\section{Preliminaries and basic properties}\label{sec:preliminaries}

\subsection{Preliminaries}\label{ssec:prelim-basics}
Recall that the standard planar Brownian bridge \( B \) is a centered Gaussian process with covariance (see \cite[p.\ 64]{borodin_salminen})
\begin{equation}\label{eq:bridge-cov}
	\E[B_sB_t^\top]=(\min\{s,t\}-st)I,\qquad 0\le s,t\le 1,
\end{equation}
where we identify the process at a given time with the appropriate column-vector and $\tau$ denotes transposing.
With simple calculations we establish the covariance structure of the process in~\eqref{eq:def-process}.
Using the independence of \( B \) and \( Z \),
\begin{equation}\label{eq:covariance}
	\begin{aligned}
		\E \ug{ \iproc \si s {\iproc \si t}^\top } & = \E \ug{B_s B_t^\top} + \E \ug{ZZ^\top}  \\
		                                           & = \ob{\min \vit{s,t} + (\sigma^2-1) st}I,
		\qquad 0\le s,t\le 1.
	\end{aligned}
\end{equation}
Note that by setting $\sigma = 0$ in \eqref{eq:covariance} we get the covariance function of a standard planar Brownian bridge, while setting $\sigma = 1$ gives us the covariance function of standard planar Brownian motion. Hence, as mentioned above, the process \( \iprocs \si \) can be viewed as a natural Gaussian interpolation between the Brownian bridge and Brownian motion.

\begin{remark}[An equivalent interpolation]
	It turns out that the presented interpolation between Brownian bridge and Brownian motion is equivalent to an another very natural way of interpolating between the two processes. Let \(W\) be a standard planar Brownian motion, and let \(B\) be an independent standard ($0$-to-$0$) planar Brownian bridge. For \(\lambda\in\cc{0,1}\), define
	\[
		Y_\lambda(t)=
		\sqrt{1-\lambda} \, W_t+\sqrt{\lambda} \, B_t,
		\qquad 0\le t\le1.
	\]
	Then \(Y_\lambda\) is a centered Gaussian process with covariance
	\begin{align}
		\E \ug{ Y_\lambda(s) Y_\lambda(t)^\top }
		 & = \ug{ (1-\lambda)\min \vit{s,t}
		+ \lambda(\min \vit{s,t} - st)} I      \\
		 & = \ob{\min \vit{s,t}-\lambda st} I.
	\end{align}
	Therefore
	$
		\iprocs \si   \jpod Y_\lambda
	$
	whenever
	$
		\lambda=1-\sigma^2.
	$
	However, the representation in~\eqref{eq:def-process}
	is more convenient for our purposes, since conditionally on the endpoint $\iproc \si 1$ the process
	clearly reduces to a planar Brownian bridge with a deterministic endpoint (though not necessarily the usual $0$).

\end{remark}

A key property of the process is isotropy,
i.e.\ for every orthogonal \( 2 \times 2 \) matrix $Q$,
the processes \( \iprocs \si \) and \( Q \iprocs \si \) are identically distributed.
Namely, it is known that both \( B \) and \( P \coloneq (tZ)_{0\le t\le1} \) are isotropic. Since they are mutually independent,
\begin{equation}
	\ob{ QB, QP } \jpod \ob{ B,  P}
\end{equation}
and the claim follows.

\subsection{Support functions and the reduction to one dimension}
We denote
\begin{equation}
	e_\theta = (\cos \theta, \sin \theta), \quad
	e_\theta^\perp = (-\sin \theta, \cos \theta), \qquad \theta \in \co{0,2\pi}.
\end{equation}
The support function of the convex hull \( K_\sigma \) in direction \( \theta \)
is given by
\begin{equation}
	M_\sigma(\theta) = \sup_{0\le t\le1} \skp{\iproc \si t}{e_\theta}.
\end{equation}
Recall Cauchy's formulas for the perimeter \cite{cauchy1832memoire, tsukerman_veomett} and area \cite{hsiung} of convex compact sets:
\begin{equation}
	\per K_\sigma = \int_0^{2\pi} M_\sigma(\theta) \D \theta, \qquad
	\area K_\sigma = \frac 12 \int_0^{2\pi} \ug{M_\sigma(\theta)^2 - M_\sigma'(\theta)^2}  \D \theta.
\end{equation}
Next, observe that due to isotropy the processes \( \skp {\iprocs \si}{e_0}  \) and
\( \skp{\iprocs \si}{e_\theta} \) have the same law for any \( \theta \). Therefore
\( \E [M_\sigma(\theta)] = \E [M_\sigma(0) \)], independently of \( \theta \). Integrating,
we arrive at
\( \E [\per K_\sigma] = 2\pi \E [M_\sigma(0)] \). Hence, we have
reduced the problem of the expected perimeter to a single dimension. The same can be done with the expected area.
The simplified formulas for the expected perimeter and area are therefore
\begin{equation}\label{eq:reduxformule}
	\E [\per K_\sigma] = 2\pi \E [M_\sigma(0)],
	\qquad
	\E [\area K_\sigma] = \pi \ob{ \E [M_\sigma(0)^2] - \E [M'_\sigma(0)^2]}.
\end{equation}

\subsection{The projected one-dimensional process}\label{ssec:projproc}
According to the previous discussion, the problem of the expected perimeter and area reduce to the one dimesional process
\begin{equation}
	\gproc x \si t = \skp{ \iproc \si t }{e_0}, \qquad 0\le t\le1,
\end{equation}
which can also be written as
\begin{equation}
	\gproc x\si t = B_t \nderi 1 + \sigma t N,
\end{equation}
where \( B \nderi 1 \) is a standard \( 0 \)-to-\( 0 \) one-dimensional Brownian bridge, and
\( N \sim \up N (0, 1) \) is an independent standard normal random variable. Then \( \E [\per K_\sigma] = 2\pi \E [M_\sigma] \), where
\( M_\sigma = \sup_{0\le t\le1} \gproc x \si t \) is the simple one-dimensional maximum. In what follows it will be convenient to use conditioning on the random endpoint
\begin{equation}
	\gproc x \si 1 \eqcolon Y = \sigma N \sim \up N(0, \sigma^2).
\end{equation}
Conditionally on \( Y=y \), the process \( \gprocs x \si \) is a one-dimensional
\( 0 \)-to-\( y \) Brownian bridge. We recall a classical result (see \cite[Theorem 2.1 and Remark 2.1]{beghin_orsingher})
on the distribution of the maximum of such a process.

\begin{lemma}\label{lma:br0ymax}
	Let $B^{0\to y}$ be a one-dimensional \( 0 \)-to-\( y \) Brownian bridge on $[0,1]$, and let
	\[
		M \nderi y = \max_{0\le t\le 1}B_t^{0\to y}.
	\]
	Then for every $m \ge 0$
	\begin{equation}\label{eq:br0ymax}
		\P(M^{(y)}\le m)=
		\begin{cases}
			0,              & m<\max\{y,0\},    \\[0.5ex]
			1-e^{-2m(m-y)}, & m\ge \max\{y,0\}.
		\end{cases}
	\end{equation}
\end{lemma}

\section{The expected perimeter (single process)}\label{sec:perimter}
In this section, we  derive a closed formula for \( \E [\per K_\sigma] \).
By~\eqref{eq:reduxformule} the problem reduces to finding the expected value of the maximum \( M_\sigma \) of the projected one-dimensional process. First, we obtain an explicit expression for the cumulative distribution function of \( M_\sigma \). To do so, we use the aforementioned
conditioning on the random endpoint \( Y \). Afterwards, by employing the tail-integral formula we derive \( \E [M_\sigma] \).

\begin{proposition}\label{prop:cdf}
	Let $0<\sigma\le 1$, and let $\Phi$ denote the standard normal distribution function. Then for every $m\ge 0$,
	\begin{equation}\label{eq:cdf}
		\P(M_\sigma\le m)
		=
		\Phi \ob{\frac m\sigma}
		-
		e^{-2(1-\sigma^2)m^2}
		\Phi \ob{\frac{(1-2\sigma^2)m}\sigma}.
	\end{equation}
	For $\sigma=0$,
	\begin{equation}\label{eq:cdf0}
		\P(M_0\le m)=1-e^{-2m^2},\qquad m\ge 0.
	\end{equation}
\end{proposition}

\begin{proof}
	Fix \( 0 < \sigma \le 1 \) and \( m \ge 0 \). Conditioning on \( Y \sim \up N(0, \si^2) \) and using Lemma~\ref{lma:br0ymax},
	\begin{align}
		\P(M_\sigma \le m) & = \int_{-\infty}^m \ob{1 - e^{-2m(m-y)}} \frac 1{\sigma \sqrt{2\pi}} e^{-y^2/2\si^2} \D y                                       \\
		                   & = \Phi\ob{\frac m\sigma} - e^{-2m^2} \int_{-\infty}^m e^{2my} \frac 1{\sigma \sqrt{2\pi}} e^{-y^2/2\si^2} \D y                  \\
		                   & = \Phi\ob{\frac m\sigma} - e^{-2m^2} e^{2\si^2m^2} \int_{-\infty}^m \frac 1{\si\sqrt2{\pi}} e^{-\frac{(y - 2\si^2m)^2}{2\si^2}} \\
		                   & = \Phi\ob{\frac m\si} - e^{-2(1-\si^2)m^2} \Phi \ob{\frac{(1-2\sigma^2)m}\sigma}.
	\end{align}

	\noindent	The claim for \( \sigma=0 \) is immediate from Lemma~\ref{lma:br0ymax}.
\end{proof}

\begin{remark}\label{re:sing1}
	Under the convention $m/0 = \infty$ and $\Phi(\infty) = 1$,~\eqref{eq:cdf0}
	can be viewed as a special case of~\eqref{eq:cdf}. With this in mind,
	throughout the paper we will not separately consider the case $\si=0$.
\end{remark}

\noindent We now  compute explicit expression for $\E [M_\sigma]$.

\begin{theorem}\label{thm:mean-max}
	For $0\le \sigma<1$,
	\begin{equation}\label{eq:mean-max}
		\E [M_\sigma]
		=
		\frac{1}{\sqrt{2\pi}}
		\ob{\sigma+\frac{\arccos\sigma}{\sqrt{1-\sigma^2}}}.
	\end{equation}
	At $\sigma=1$, the continuous extension is
	\begin{equation}
		\E [M_1]=\sqrt{\frac{2}{\pi}}.
	\end{equation}
\end{theorem}

\begin{proof}
	Fix \( 0 \le  \sigma < 1 \).
	We use the tail-integral formula for the expectation and use Proposition~\ref{prop:cdf}.
	We split the result into two integrals which we evaluate separately. It holds that
	\begin{align}
		\E [M_\sigma] & = \int_0^\infty \P(M_\sigma > m) \D m                                                                     \\
		              & = \int_0^\infty \ob {1 - \Phi\ob{\frac m\si} + e^{-2(1-\si^2)m^2} \Phi\ob{\frac{(1-2\si^2)m}{\si}} } \D m \\
		              & = \int_0^\infty \ob{1 - \Phi\ob{\frac m\si}} \D m
		+ \int_0^\infty e^{-2(1-\si^2)m^2} \Phi\ob{\frac{(1-2\si^2)m}\si} \D m                                                    \\ &\eqcolon I_1(\sigma) + I_2(\sigma).
	\end{align}

	We evaluate \( I_1(\sigma) \) first.
	Using the substitution \( u = m/\si \) and then integrating by parts gives
	\begin{equation}
		I_1(\sigma) = \sigma \int_0^\infty (1 - \Phi(u)) \D u
		= \si \int_0^\infty u\varphi(u) \D u,
	\end{equation}
	which is half the expectation of a folded normal variable. Therefore
	\begin{equation}\label{eq:I1}
		I_1(\sigma) = \frac \si{\sqrt{2\pi}}.
	\end{equation}

	\noindent 	Using~\cite[1,010.2]{owen}
	we get for $a>0$ and $b\in\R$,
	\begin{equation}\label{eq:Jab}
		\int_0^\infty e^{-a x^2}\Phi(bx) \D x = \frac{\sqrt{\pi}}{4\sqrt{a}} + \frac{1}{2\sqrt{\pi a}} \arctan\ob{\frac{b}{\sqrt{2a}}}.
	\end{equation}
	We can evaluate \( I_2(\sigma) \) as~\eqref{eq:Jab} where
	\begin{equation}
		a=2(1-\sigma^2),\qquad b=\frac{1-2\sigma^2}{\sigma},
	\end{equation}
	yielding
	\begin{equation}\label{eq:I2raw}
		I_2(\sigma)
		=
		\frac{1}{2\sqrt{2\pi(1-\sigma^2)}}
		\ob{
			\frac{\pi}{2}
			+
			\arctan\ob{\frac{1-2\sigma^2}{2\sigma\sqrt{1-\sigma^2}}}
		}.
	\end{equation}
	This can be further refined by setting \( \si = \cos \theta \) for some \( \theta \in \oc{0, \pi/2} \), i.e.
	\begin{equation}
		\frac{1 - 2\si^2}{2\si\sqrt{1-\si^2}} = \frac{-\cos 2\theta}{ \sin 2\theta} = -\ctg 2 \theta
	\end{equation}
	and then
	\begin{equation}
		\frac \pi 2 + \arctan (-\ctg 2 \theta) = 2 \theta = 2\arccos \sigma.
	\end{equation}
	We arrive at
	\begin{equation}
		I_2(\sigma) = \frac 1 {\sqrt{2\pi}} \frac{\arccos \si}{\sqrt{1 - \sigma^2}},
	\end{equation}
	which combined with~\eqref{eq:I1} proves~\eqref{eq:mean-max}.

	To formally justify the continuous extension as $\si \uparrow 1$, first note that by continuous mapping theorem
	$M_\si \konvd  M_1$ as $\left. \si \uparrow 1\right.$. This together with the uniform
	boundedness of second moments (see
	Proposition~\ref{prop:M2}) show the uniform integrability
	of $\vit{M_\si \st 0 \le \si < 1}$ and $\E [M_\si] \to \E [M_1]$
	as $\left. \si \uparrow 1\right.$.
\end{proof}

As a direct concequence of Theorem \ref{thm:mean-max} we get the following.
\begin{corollary}
	For \( 0 \le \si < 1 \),
	\begin{equation}
		\E [\per K_\sigma] = \sqrt{2\pi} \ob{\sigma+\frac{\arccos\sigma}{\sqrt{1-\sigma^2}}},
	\end{equation}
	and for \( \sigma=1 \), \( \E [\per K_1] = \sqrt{8\pi} \).
\end{corollary}

\begin{remark}
	As already commented, the corollary recovers some of the classical results:
	\begin{itemize}
		\item for $\sigma=0$, the expected perimeter of the convex hull of a planar Brownian bridge (see~\cite{goldman}):
		      \[
			      \E [\per K_0]
			      =
			      \sqrt{\frac{\pi^3}{2}};
		      \]
		\item for $\sigma=1$, the expected perimeter of the convex hull of a planar Brownian motion (see~\cite{letac_takacs}):
		      \[
			      \E [\per K_1] =\sqrt{8\pi}.
		      \]
	\end{itemize}
\end{remark}

\section{The expected area (single process)}\label{sec:area}
Recall the Cauchy's area formula from~\eqref{eq:reduxformule}. To compute the second moment
of the maximum, we proceed
analogously as
in Theorem~\ref{thm:mean-max}.
Second moment of the derivative of the maximum, on the other hand, is more subtle---\(M'_\sigma(0)\) is in fact the $y$-coordinate
of the point with the maximal $x$-coordinate.
The approach there is to use conditioning
on the maximizing time, i.e.\ the time at which the maximum in the direction of the $x$-axis is achieved. This is the general approach outlined in \cite[\textsection 5.2]{majumdar}.

\subsection{The second moment of the maximum}\label{ssec:EM2}
In this subsection, we compute $\E [M_\si^2]$.

\begin{proposition}\label{prop:M2}
	For $0\le\sigma\le1$,
	\begin{equation}\label{eq:EM2}
		\E [M_\sigma^2] =\frac{1+\sigma^2}{2}.
	\end{equation}
\end{proposition}

\begin{proof}
	Recalling Proposition~\ref{prop:cdf}, the proof is analogous
	to that of Theorem~\ref{thm:mean-max}. We have that
	\begin{align}
		\E [M_\si^2] & = 2\int_0^\infty m\P(M_\si > m) \D m                        \\
		             & = 2 \int_0^\infty m \ob{1 - \Phi\ob{\frac m\si}} \D m
		+ 2 \int_0^\infty m e^{-2(1-\si^2)m^2} \Phi\ob{\frac{(1-2\si^2)m}\si} \D m \\& \eqcolon I_1(\sigma) + I_2(\sigma).
	\end{align}

	\noindent	By using the substitution  $u = m/\si$ and integration by parts formula, we conclude that
	\begin{align}
		I_1(\si) & = 2 \si^2 \int_0^\infty u(1 - \Phi(u)) \D u                   =  {\si^2} \int_0^\infty u^2 \varphi(u) \D u = \frac {\si^2}2.
	\end{align}
	To evaluate $I_2(\si)$, we use~\cite[1,011.2]{owen}
	which immediately implies
	\begin{equation}\label{eq:Kab}
		\int_0^\infty x e^{-a x^2}\Phi(bx)\D x = \frac{1}{4a} + \frac{b}{4a\sqrt{2a+b^2}}.
	\end{equation}
	Setting
	\begin{equation}
		a =2(1-\sigma^2),
		\qquad
		b=\frac{1-2\sigma^2}{\sigma},
	\end{equation}
	we conclude (for \( a=0 \), note integrability for \( b < 0 \) and use a dominated convergence argument)
	$
		I_2(\si) = 1/2,
	$
	arriving at~\eqref{eq:EM2}.

\end{proof}

\subsection{The distribution of the maximizing time}
Let $y_\sigma$ be the $y$-coordinate process of
$\iprocs \si$, i.e.
\begin{equation}
	\gproc y \si t = \skp {\iproc \si t} e_0^\perp,
\end{equation}
and let $\tau_\si$ be the maximizing time of $\gprocs x \si$ (it is a.s.\ unique---see Lemma~\ref{lma:uniqueargmax}), i.e.
\begin{equation}
	\tau_\si = \argmax\limits_{0\le t\le1} \gproc x \si t.
\end{equation}
Then, continuing on our previous discussion, we can write
$M'_\si(0) = \gproc y \si {\tau_\si}$.
In order to compute $\E[M'_\si(0)^2]$, our strategy is to  first condition on the time $\tau_\si$ (noting that $\gprocs x \si$ and $\gprocs y \si$ are independent),
and then compute explicitly the density function of $\tau_\si$.
First, we compute the density of
\begin{equation}
	\tau \nderi y  = \argmax\limits_{0\le t\le1} B_t^{0 \to y},
\end{equation}
where $B^{0 \to y}$ is a one-dimensional $0$-to-$y$ Brownian bridge. To do so, we start with a different perspective. The time
$\tau \nderi y$ can be viewed as the maximizing time $\tau_W$ of a one-dimensional Brownian motion $W$
conditionally on $W(1) = y$. From now on we denote the density of a variable $X$ as $f_X$.
The joint density of $(\tau_W, W(1))$ is given by~\cite[eq.~(2.12)]{riedel} and it can be written as
\begin{equation}
	f_{(\tau_W, W(1))}(t, y) = 2\varphi(y) \ug{
		y \sqrt{\frac t {1-t}} \varphi\ob{{y \sqrt{\frac t{1-t}}}} + (1-y^2)\ol \Phi\ob{y \sqrt{\frac t{1-t}}}
	},\qquad y\ge0,
\end{equation}
where $\ol \Phi = 1 - \Phi$. The density of $\tau \nderi y$
is then rather simple to obtain from
\begin{equation}
	f_{\tau \nderi y} = \frac{f_{(\tau_W, W(1))}}{f_{W(1)}},
\end{equation}
as $f_{W(1)} \equiv \varphi$. When $y < 0$, consider the reversed process
$$t \mapsto y - B^{0 \to y}_{1-t},$$ which is
a one-dimensional Brownian bridge from $0$ to $-y$.
Consequently, $\tau \nderi y \jpod 1 -\tau \nderi{-y}$ and
$f_{\tau \nderi y}(t) = f_{\tau \nderi {-y}}(1-t)$. We have proven the
following proposition.

\begin{proposition}\label{prop:denstauy}
	The density of the maximizing time $\tau \nderi y$ for a fixed
	$y \ge 0$ is given by
	\begin{equation}\label{eq:denstauy}
		f_{\tau \nderi y}(t) = 2 \ug{
			y \sqrt{\frac t {1-t}} \varphi\ob{{y \sqrt{\frac t{1-t}}}} + (1-y^2)\ol \Phi\ob{y \sqrt{\frac t{1-t}}}
		},
	\end{equation}
	and when $y < 0$  is given by
	$f_{\tau \nderi y}(t) = f_{\tau \nderi {-y}}(1-t)$.
\end{proposition}

We now compute the density of $\tau_\si$.
Clearly, as $Y$ has density $\varphi(y/\sigma)/\si$
for $\si > 0$ (note that $\si=0$ makes a trivial case as $f_{\tau_0} = f_{\tau \nderi 0}$),
\begin{align}
	f_{\tau_\si}(t) & = \int_0^\infty \frac 1 \sigma f_{\tau \nderi y}(t) \varphi(y/\si) \D y +
	\int_0^\infty \frac 1\si f_{\tau \nderi {y}}(1-t) \varphi(t/\si) \D y                       \\
	                & \eqcolon A_\si(t) + A_\si(1-t).
\end{align}

\noindent To compute $A_\sigma(t)$, fix $0<t<1$ and set
\begin{equation}\label{eq:def-ab}
	a_t=\sqrt{\frac{1-t}{t}}, \qquad
	b_t=\sqrt{\frac{t}{1-t}}=\frac{1}{a_t}.
\end{equation}
Using Proposition \ref{prop:denstauy}, we get
\begin{equation}\label{eq:Asigma-expanded}
	A_\sigma(t)
	= \frac{2}{\sigma}\int_0^\infty \ug{y b_t\,\varphi(a_t y)
		+
		(1-y^2)\ol \Phi (a_t y) }
	\varphi \ob{\frac{y}{\sigma}} \D y.
\end{equation}
The problem reduces to evaluating three Gaussian integrals, namely for $k > 0$ we need to compute
\begin{equation}
	\begin{aligned}\label{eq:I0I1I2}
		I_0(k) & =\int_0^\infty \varphi(x)\,\ol \Phi(kx) \D x,
		\\
		I_1(k) & =\int_0^\infty x \varphi(x)\varphi(kx)\D x,
		\\
		I_2(k) & =\int_0^\infty x^2\varphi(x)\ol \Phi (kx)\D x.
	\end{aligned}
\end{equation}
Using \cite[\textsection 4.3, eq.~2]{ng_geller}, we have
\begin{equation}
	I_0(k) = \frac{1}{4} - \frac{1}{2\pi}\arctan k.
\end{equation}
For $I_1(k)$, we directly use the antiderivative and obtain
\begin{equation}
	I_1(k) = \frac{1}{2\pi(1 + k^2)}.
\end{equation}
Finally, since
\begin{equation}
	x^2 \varphi(x) = \varphi(x) - \frac{\D}{\D x}(x \varphi(x)),
\end{equation}
we get
\begin{equation}
	I_2(k) = I_0(k) - kI_1(k) = \frac{1}{4} - \frac{1}{2\pi}\arctan k - \frac{k}{2\pi(1 + k^2)}.
\end{equation}

Set now $y=\sigma x$ in \eqref{eq:Asigma-expanded}.
Since $\varphi(y/\si)/\si \D y=\varphi(x)\D x$,
we obtain
\begin{align}
	A_\sigma(t)
	 & =
	2\int_0^\infty
	\ug{
		(\sigma x)b_t \varphi(\sigma a_t x)
		+
		(1-\sigma^2x^2)\ol \Phi (\sigma a_t x)
	}
	\varphi (x)\D x \\
	 & =
	2\ug{
		\sigma b_t I_1(\sigma a_t)
		+
		I_0(\sigma a_t)
		-
		\sigma^2 I_2(\sigma a_t)
	}.
	\label{eq:Asigma-I}
\end{align}
Using \eqref{eq:I0I1I2}, we conclude the following.

\begin{proposition}
	For $0 < t < 1$ and $0 < \sigma \le 1$, the quantity $A_\si(t)$ is given by
	\begin{equation}\label{eq:Asigma-closed}
		A_\sigma(t)
		=
		\frac{\sigma}{\pi}\sqrt{\frac{t}{1-t}}
		+
		(1-\sigma^2)\ug{
			\frac12-\frac{1}{\pi}
			\arctan \ob{\sigma\sqrt{\frac{1-t}{t}} }
		}.
	\end{equation}
\end{proposition}

We can now evaluate
$f_{\tau_\si}(t) = A_\si(t) + A_\si(1-t)$, using, for the arctangents, the additive formula
\begin{equation}
	\arctan u + \arctan v = \arctan \ob{\frac{u+v}{1-uv}},
\end{equation}
where in our case $uv = \si^2 \le 1$ satisfies the condition. For $uv=1$ we use the conventions $\cdot/0 = \infty$ and $\arctan \infty = \pi/2$---this is why the formula~\eqref{eq:fsigma-closed} will be valid even for $\sigma=1$.
Finally, we obtain the following.
\begin{theorem}\label{thm:fsigma}
	For $0<\sigma\le 1$, the maximizing time $\tau_\sigma$ has density
	\begin{equation}\label{eq:fsigma-closed}
		f_{\tau_\si}(t)
		=
		\frac{\sigma}{\pi\sqrt{t(1-t)}}
		+
		(1-\sigma^2)\ug{
			1-\frac{1}{\pi}
			\arctan \ob{
				\frac{\sigma}{(1-\sigma^2)\sqrt{t(1-t)}}
			}
		},
		\qquad 0<t<1.
	\end{equation}
\end{theorem}

\begin{remark}
	Considering $\si=0,1$ we once again reconstruct classical results:
	\begin{equation}
		f_{
		\tau_0
		}(t) = 1, \quad f_{\tau_1}(t) = \frac1{\pi \sqrt{t(1-t)}}, \qquad
		0 < t < 1,
	\end{equation}
	i.e. for a Brownian bridge the distribution is uniform (see e.g.~\cite[IV.24]{borodin_salminen}).
	In fact, the result for the Brownian bridge was already derived as a special case in
	Proposition~\ref{prop:denstauy}, but we can also see that it is the limit
	of $f_{\tau_\si}$ as $\si \downarrow 0$.
	For Brownian motion, the distribution is the famous arcsine law (see e.g.\ \cite[Thm.~5.26]{morters_peres}).
\end{remark}

\subsection{The derivative component and final formula}
We now move to  $\E [M'_\si(0)^2]$. Recall that $M'_\si(0) = \gproc y \si {\tau_\si}$.
When $\tau_\si$ is replaced by a fixed $t$, from~\eqref{eq:covariance} it follows that
\[ \E [\gproc y \si t^2] = t - (1 - \si^2) t^2.\]
By
conditioning on $\tau_\si$,
\begin{align}\label{eq:EMderiv21}
	\E [M'_\si(0)^2]= \int_0^1 (t - (1-\si^2)t^2)  f_{\tau_\si} (t) \D t.
\end{align}

\begin{proposition}\label{prop:Mprime2}
	For $0\le\sigma\le1$,
	\begin{equation}\label{eq:EMderiv2}
		\E [M_\sigma'(0)^2]
		=
		\frac{3\sigma^3+\sigma^2+\sigma+1}{6(\sigma+1)}.
	\end{equation}
\end{proposition}

\begin{proof}
	Rearranging terms in~\eqref{eq:EMderiv21}, for $0 < \sigma \le 1$,
	\begin{align}\label{eq:EMderiv22}
		\E [M'_\si(0)^2] & = \int_0^1 (t - (1-\si^2)t^2)  f_{\tau_\si} (t) \D t                   \\
		                 & = \si^2 \int_0^1 t f_{\tau_\si}(t) \D t +
		(1-\si^2) \int_0^1 t(1-t) f_{\tau_\si}(t) \D t                                            \\
		                 & = \si^2 \E \tau_\si + 2 (1 - \si^2 ) \int_0^1 t(1-t) A_\sigma(t) \D t.
	\end{align}
	Due to the symmetry of $f_{\tau_\si}$ around $1/2$, it is easy to
	see that $\E [\tau_\si] = 1/2$. The second part can be futher separated into three integrals.
	The first two are elementary:
	\[
		\int_0^1 t(1-t) \D t=\frac16,
		\qquad
		\int_0^1 t^{3/2}(1-t)^{1/2} \D t= \up B \ob{\frac 52, \frac 32} = \frac{\pi}{16},
	\]
	while the final
	\[
		\int_0^1
		t(1-t)
		\arctan \ob{\sigma\sqrt{\frac{1-t}{t}}} \D t
		=
		\frac{\pi\sigma(4\sigma^2+9\sigma+3)}{48(1+\sigma)^3}
	\]
	is computed in Lemma~\ref{lem:a2}. Combining these results we arrive at~\eqref{eq:EMderiv2}.
	The remaining case $\si =0$ is trivial.
\end{proof}

With~\eqref{eq:reduxformule}, \eqref{eq:EM2} and~\eqref{eq:EMderiv2},
we have arrived at the desired formula for the expected area of
$K_\si$.

\begin{corollary}
	For $0 \le \si \le 1$,
	\begin{equation}\label{eq:Earea}
		\E [\area K_\si] = \frac \pi 3 \cdot \frac{\si^2 + \si + 1}{\si + 1}.
	\end{equation}
\end{corollary}

\begin{remark}
	Once again we can reconstruct the classical formulas for the Brownian bridge and
	Brownian motion respectively (see \cite{elbachir, majumdar}), i.e.
	\begin{equation}
		\E [\area K_0] = \frac \pi 3, \qquad
		\E [\area K_1] = \frac \pi 2.
	\end{equation}
\end{remark}

\newcommand{\joi}{{\text{joint}}}
\section{Multiple processes}\label{sec:multiple_processes}
In this section, we consider multiple independent Gaussian-endpoint Brownian bridges \( X_1\), \( X_2\), \( \ldots, X_n \)
with parameters \( \si_1, \si_2, \ldots, \si_n \in \cc{0,1} \).
We always use \( j \) in the subscript when referring to the \( j \)-th process and
\enquote{joint} when referring to their joint maximum in the direction of the \( x \)-axis,
the corresponding maximizing time or the joint convex hull.
The processes are jointly isotropic, so we could use the same approach
for the expected perimeter of the joint convex hull, calculating the expected joint maximum
in the $x$-direction. We have that
\begin{align}
	\E [M_\joi] & = \int_0^\infty \P(M_\joi > m) \D m                          = \int_0^\infty 1 - \prod_{j=1}^n \P(M_{\si_j} \le m) \D m.
\end{align}
With~\eqref{eq:cdf} in mind, establishing a closed formula
would mean evaluating integrals of the form
\begin{equation}\label{eq:Habc-intro}
	\int_0^\infty e^{-a_0x^2} \prod_{j=1}^n \Phi(a_j x) \D x.
\end{equation}
These are the so-called orthant probabilities, and it is well known
that no closed forms are known for $n \ge 3$,
aside from special cases. For \( n=2 \), there is a closed formula
given in~\cite[2,010.6]{owen}, and in Subsection~\ref{ssec:perim2}
we present  a closed formula for the expected perimeter
in this case.

For the expected area of the joint convex hull, calculating
\( \E [M_\joi ^2] \) would mean evaluating integrals of the form
\begin{equation}
	\int_0^\infty x e^{-a_0x^2} \prod_{j=1}^n \Phi(a_j x) \D x.
\end{equation}
Integrating by parts this is immediately reduced to
integrals of the kind in eq.~\eqref{eq:Habc-intro}.
Moreover, the \enquote{degree} \( n \) will be reduced by \( 1 \),
making the computation possible even for \( n=3 \).
The derivative component is again more complicated.
If we denote by \( J \) the random index of the process
which achieves the joint maximum, \( M_\joi'(0) = y_J(\tau_J) \).
By conditioning both on the index \( J \)
and the joint maximizing time \( \tau_J = \tau_\joi \),
we can write
\begin{equation}\label{eq:a21}
	\E [M'_\joi(0)^2] = \sum_{j=1}^n \int_0^1 (t - (1 -\si_j^2)t^2) f_{(J, \tau_\joi)}(j, t) \D t,
\end{equation}
where \( f_{(J, \tau_\joi)} \) is the mixed mass-density
function of the vector. It can be calculated by  considering the joint maximum. Namely,
\begin{equation}\label{eq:a22}
	f_{(J, \tau_\joi)}(j, t ) = \int_0^\infty f_{(\tau_{\si_j}, M_{\si_j})}(t,m) \prod_{k \neq j} \P(M_{\si_k} \le m) \D m.
\end{equation}
The reasoning is that the \( j \)-th process has to achieve maximum \( m \) while
all others must remain below it. This general approach was seen in~\cite{majumdar} and~\cite{sebek}.
The density \( f_{(\tau_{\si_j}, M_{\si_j})} \) can, in turn, be obtained by conditioning on the random endpoint
of the projected \(j \)-th process, i.e.
\begin{equation}\label{eq:a23}
	f_{(\tau_{\si_j}, M_{\si_j})}(t,m) = \int_{-\infty}^m f_{(\tau \nderi y, M \nderi y)}(t,m) \cdot \frac 1{\si_j} \varphi\ob{\frac y{\si_j}} \D y.
\end{equation}
Finally, \( f_{(\tau \nderi y, M \nderi y)} \) is obtained by dividing the formula~\cite[eq.\ (2.1)]{riedel} for the joint density of \( (W(1), \tau_W, M_W) \) (respectively the endpoint, maximizing time and maximum of a one-dimensional Brownian motion supported on \( \cc{0,1} \)) with \( \varphi(y) \), once again
considering the \( 0 \)-to-\( y \) bridge as Brownian motion conditionally on \( W(1)=y \). Said lattermost formula is, in our notation,
\begin{equation}
	f_{(W(1), \tau_W, M_W)} (y,t,m) =
	\frac{m(m-y)}{\pi \ug{t(1-t)}^{3/2}}
	\exp \ug{ -\frac{m^2}{2t} - \frac{(m-y)^2}{2(1-t)}  }.
\end{equation}
Ultimately we will see that there appears to be no closed formula even for \( n=2 \). We continue with this question in Subsection~\ref{ssec:area2}

Instead of considering a fixed number \( n \) of processes,
we could consider the asymptotics as \( n \ub \).
This problem was solved in a much more general setting in~\cite{davydov},
and we briefly apply the results in Subsection~\ref{ssec:davydov}.

\subsection{Expected perimeter (two independent processes)}\label{ssec:perim2}
We consider two independent processes \( X_1 \) and \( X_2 \) with
parameters \( \si_1, \si_2 \in \cc{0,1} \). Due to isotropy, the joint convex
hull
\begin{equation}
	K_{\si_1, \si_2} \coloneq K_\joi = \hull \ob{ \vit{ X_1(t) \st 0 \le t \le 1 }  \cup \vit{X_2(t) \st 0 \le t \le 1}}
\end{equation}
has expected perimeter
\begin{align}
	\E [\per K_\joi] =  2\pi \E [M_\joi] & = 2\pi \int_0^\infty \P(M_\joi > m) \D m                                                     \\
	                                     & = 2\pi \int_0^\infty 1 - \P(M_{\si_1} \le m) \P(M_{\si_2} \le m) \D m \label{eq:perim2form}.
\end{align}
Due to~\eqref{eq:cdf}, we see that the product of probabilities turns into a sum of \( 4 \) parts, and likewise
we split the integral into \( 4 \) parts. Three of them are of the form
\begin{equation}\label{eq:Gabc}
	G(a, b,c) = \int_0^\infty e^{-ax^2} \Phi(bx) \Phi(cx) \D x.
\end{equation}
By an application of~\cite[2,010.6]{owen} for all \( a > 0 \) and \( b,c\in\R \),
\begin{equation}\label{eq:Gabcsol}
	\begin{aligned}
		G(a,b,c) = \frac{\sqrt \pi}{8 \sqrt a} + \frac 1 {4\sqrt{\pi a}} \bigg(\! & {} \arctan \frac b{\sqrt{2a}}                                                    \\
		                                                                          & + \arctan \frac c{\sqrt{2a}} + \arctan \frac{bc}{\sqrt{2a(2a+b^2+c^2)}} \bigg) .
	\end{aligned}
\end{equation}
The other integral type which will show up in~\eqref{eq:perim2form} is
\begin{equation}\label{eq:I1bc}
	\begin{aligned}
		\int_0^\infty 1 - \Phi\ob{bx}\Phi\ob{cx} \D x
		 & = b \int_0^\infty x \varphi\ob{bx} \Phi\ob{c x} \D x+
		c \int_0^\infty x \Phi\ob{bx} \varphi\ob{cx} \D  x                              \\
		 & = \frac1{bc \sqrt{8\pi}} \ob{ b + c + \sqrt{b^2 + c^2} } \eqcolon I_1(b,c) ,
	\end{aligned}
\end{equation}
which is solved first by integration by parts (\( \up d v = \up d x \)) and then applying~\eqref{eq:Kab}.
Note that the formula holds, with the convention \( \Phi(\infty) = 1 \), even when \( b=\infty \) or \( c=\infty \).
Thus using~\eqref{eq:cdf} to evaluate~\eqref{eq:perim2form}, we get
\begin{equation}\label{eq:Eperim2proc}
	\begin{aligned}
		\E [M_\joi]
		= & I_1\ob{\frac1{\si_1},\frac1{\si_2}}  + G\ob{2(1-\si_2^2), \frac 1{\si_1}, \frac{1-2\si_2^2}{\si_2}  }                             + G\ob{ 2(1-\si_1^2), \frac 1{\si_2}, \frac{1-2\si_1^2}{\si_1} } \\
		  & - G\ob{ 2(2 - \si_1^2 - \si_2^2), \frac{1 - 2\si_1^2}{\si_1}, \frac{1 - 2\si_2^2}{\si_2} },
	\end{aligned}
\end{equation}
a complicated but closed solution. We have proven the following theorem.

\begin{theorem}\label{tm:Eperim2proc}
	The expected value  \( \E [\per K_{\si_1, \si_2}] \) is obtained by multiplying~\eqref{eq:Eperim2proc} with \( 2\pi \).
\end{theorem}

\begin{remark}\label{re:Gabc}
	The edge cases \( \sigma_j \in \vit{0,1} \) seemingly introduce singularities.
	As noted earlier, \( I(\sigma_1, \sigma_2) \) is never problematic.
	When \( b=\infty \) or \( c=\infty \) in~\eqref{eq:Gabc} (corresponding to \( \sigma_j = 0 \)), the integral reduces
	to~\eqref{eq:Jab}. In fact,~\eqref{eq:Jab} is a special case of~\eqref{eq:Gabcsol} as can be seen directly
	through the formulas or through a monotone convergence argument. If \( a=0 \) (corresponding to \( \sigma_j = 1 \)),
	the formula~\eqref{eq:Gabcsol} does not make sense at first. However, it is easy to see that the integral
	in~\eqref{eq:Gabc} will converge as long as \( b < 0 \) or \( c < 0 \), as will always be the case in our
	application. Then we can use a dominated convergence argument to show that its value is indeed the limit of~\eqref{eq:Gabcsol}
	as \(a \downarrow 0 \).
	It is likely easier to evaluate~\eqref{eq:Gabc} with \( a=0 \) with the same integration by parts used
	to evaluate \( I_1(b,c) \). Indeed, assuming \( b < 0 \) and using~\eqref{eq:Kab},
	\begin{equation}\label{eq:Gabckadaje0}
		\begin{aligned}
			\int_0^\infty \Phi(bx) \Phi(cx) \D x & =
			- \int_0^\infty bx \varphi(bx) \Phi(cx) \D x - \int_0^\infty cx\Phi(bx)\varphi(cx) \D x          \\
			                                     & = \frac{-1}{bc\sqrt{8\pi}} \ob{b + c + \sqrt{b^2 + c^2}}.
		\end{aligned}
	\end{equation}
\end{remark}

\begin{remark}
	The formula in~\eqref{eq:Eperim2proc} can be used to reconstruct more known results. Suppose
	\( \sigma_1 = 0 \) and \( \sigma_2 = 1 \), i.e.\ we are looking at an independent Brownian bridge and Brownian motion.
	We have \( I_1(1/\si_1, 1/\si_2) = 1/\sqrt{2\pi} \) and \( G(0, \infty, -1) = 1/\sqrt{2\pi} \) as
	seen in the proof of~\eqref{eq:I1}. The rest is found via~\eqref{eq:Jab}, i.e.
	\begin{equation}
		G(2,1,\infty) = \frac{\sqrt \pi}{4\sqrt 2} + \frac 1{2 \sqrt{2\pi}} \arctan \frac 12, \qquad
		G(2, \infty, -1) = \frac{\sqrt \pi}{4\sqrt 2} - \frac 1{2 \sqrt{2\pi}} \arctan \frac 12.
	\end{equation}
	Summing and multiplying by \( 2\pi \) we get the expected perimeter
	\begin{equation}
		\E [\per K_{0,1}] = \sqrt{2\pi} \ob{2 + \arctan \frac 12},
	\end{equation}
	as already seen in \cite[Theorem 1.1]{sebek}.
	Also, we can consider two independent Brownian motions (\( \si_1=\si_2=1 \))
	which, using~\eqref{eq:Gabckadaje0}, yields the result
	\begin{equation}
		\E [\per K_{1,1}] = 4\sqrt\pi,
	\end{equation}
	as seen in~\cite[p.\ 991]{majumdar}. For two independent Brownian bridges (\( \si_1=\si_2=0 \))
	the calculations are the simplest as they involve \( b=c=\infty \) in~\eqref{eq:Gabc}
	and we arrive at
	\begin{equation}
		\E [\per K_{0,0}] = \pi^{3/2} \ob{\sqrt 2 - \frac 12},
	\end{equation}
	as seen in~\cite[p.\ 994]{majumdar}.
\end{remark}

\subsection{Expected area (two processes)}\label{ssec:area2}
In this section, we consider the expected area of the joint convex hull spanned by two independent
processes with parameters \( \si_1, \si_2 \in \cc{0,1} \). We start with the
expectation \( \E [M_\joi'(0)^2] \), continuing on the discussion in the introduction of
the section. First we evaluate the density \( f_{(\tau_{\si_j}, M_{\si_j})} \).
Through completing the square on the right side of~\eqref{eq:a23}, we get an expression of the form
\begin{equation}\label{eq:denstausiMsi}
	f_{(\tau_{\si_j},M_{\si_j})}(t, m) = C_{0,j}\ob{ m^2 e^{-C_{1, j} m^2} \Phi(C_{2,j} m) + \frac1{\sqrt{2\pi} C_{2,j}} m e^{-C_{3,j} m^2}},
\end{equation}
for all $m \ge 0$ and $t \in \cc{0,1}$ where
\begin{equation}
	\begin{aligned}
		C_{0,j} & = \sqrt{\frac2\pi} \cdot \frac{1-\sigma_j^2}{t^{3/2} (t\sigma_j^2 + 1 - t)^{3/2}}, &
		C_{1,j} & = \frac{1}{2t(t\sigma_j^2 + 1 - t)},                                                 \\
		C_{2,j} & = \frac{1-\sigma_j^2}{\sigma_j} \sqrt{\frac{1-t}{t\sigma_j^2 + 1 - t}},            &
		C_{3,j} & = \frac{1}{2} \ob{ \frac{1}{t} + \frac{1}{\sigma_j^2} - 1 }
	\end{aligned}
\end{equation}
are functions in $t$ for a fixed \( j \in \vit{1,2} \).
Before moving on to
evaluating~\eqref{eq:a22}, consider the integral type
\begin{equation}
	H(a,b,c) = \int_0^\infty x^2 e^{-ax^2} \Phi(bx) \Phi(cx) \D x.
\end{equation}
Using integration by parts (\( \up d v = x e^{-ax^2} \up d x \)) and
denoting~\eqref{eq:Kab} as \( K(a,b) \), we obtain
\begin{equation}\label{eq:Habc}
	\begin{aligned}
		H(a,b,c) = \frac1{2a} \ug{ G(a,b,c) + \frac b{\sqrt{2\pi}} K\ob{ a + \frac{b^2}2,c}
			+ \frac c{\sqrt{2\pi}} K\ob{a + \frac {c^2}2,b}}.
	\end{aligned}
\end{equation}
Fix  $j \in \vit{1,2}$ and let $k$ be the other index.
Using~\eqref{eq:denstausiMsi} and~\eqref{eq:cdf} we can
evaluate~\eqref{eq:a22} as
\begin{equation}
	\begin{aligned}
		f_{(J, \tau_\joi)}(j, t ) = & C_{0,j} [ H(C_{1,j}, C_{2,j}, C_{4,k})  -
		H(C_{1,j} + C_{5,k}, C_{2,j}, C_{6,k} )                                                                                                          \\
		                            & + \frac1{C_{2,j}\sqrt{2\pi}} K(C_{3,j},C_{4,k})          -\frac1{C_{2,j}\sqrt{2\pi}} K(C_{3,j} + C_{5,k}, C_{6,k})
		],
	\end{aligned}
\end{equation}
where the constants
\begin{equation}
	C_{4, k} = \frac1{\si_k}, \qquad
	C_{5,k} = 2(1-\si_k^2), \qquad
	C_{6,k} = \frac{1-2\si_k^2}{\si_k}
\end{equation}
correspond to those in~\eqref{eq:cdf}.
Integrating the result in~\eqref{eq:a21} leads to an expression
which appears not to have a closed form, as it is much more complicated than that in
Lemma~\ref{lem:a2} and the substitution would not leave us with a rational function.
This is the case even when $\si_1 \in \vit{0,1}$ or $\si_2\in\vit{0,1}$. When
both $\si_1, \si_2 \in \vit{0,1}$, the results are closed formulas known
from~\cite{majumdar} and~\cite{sebek}.

We move on to calculating \( \E [M_\joi^2] \).
We use the tail-integral formula as in Subsection~\ref{ssec:EM2}. The calculation is
a variant of that in the previous Subsection~\ref{ssec:perim2}, the difference being
the factor \( 2x \) in the integrand.
The first integral which will show up, corresponding to~\eqref{eq:I1bc}, is of the form
\begin{equation}\label{eq:mjoi21}
	\begin{aligned}
		\int_0^\infty x(1-\Phi(bx)\Phi(cx)) \D x =
		\frac 1{\sqrt{8\pi}} \ug{bH(b^2/2, c, \infty) + cH(c^2/2, b, \infty)},
	\end{aligned}
\end{equation}
where we used the integration by parts
involving \( \up dv = x \up dx \).
The expression \( H(a, b, \infty) \) can, starting from~\eqref{eq:Habc}, be further simplified as
\begin{equation}
	H(a, b, \infty) = \frac1{2a}\ug{J(a,b) + \frac b{\sqrt{2\pi}(2a + b^2)} },
\end{equation}
where \( J(a,b)=G(a,b,\infty) \) denotes~\eqref{eq:Jab}
and we use \( K(a,\infty) = 1/2a \) and \( c^2K(c+a,b) \to 0 \) as \( c \ub \).
Using~\eqref{eq:Jab}, eq.~\eqref{eq:mjoi21} can then also be written as
\begin{equation}\label{eq:mjoi21simp}
	\begin{aligned}
		\int_0^\infty x(1-\Phi(bx)\Phi(cx)) \D x
		 & = \frac1{4\pi}\ug{ \frac1{b^2}\ob{\frac\pi2 + \arctan \frac bc}
		+ \frac 1{c^2} \ob{\frac\pi2 + \arctan \frac cb} + \frac 1{bc}}    \\
		 & \eqcolon I_2(b,c).
	\end{aligned}
\end{equation}
The final integral type we will encounter, corresponding to~\eqref{eq:Gabc} in
Subsection~\ref{ssec:perim2}, is
\begin{equation}\label{eq:Labc}
	L(a,b,c) = \int_0^\infty xe^{-ax^2}\Phi(bx)\Phi(cx) \D x.
\end{equation}
As discussed previously we use integration by parts
with \( \up dv = xe^{-ax^2} \up dx \) to obtain
\begin{equation}\label{eq:Labcsol}
	\begin{aligned}
		L(a,b,c) & = \frac1{8a} + \frac b{a \sqrt{8\pi}} J\ob{a + \frac{b^2}2 , c}
		+ \frac c{a\sqrt{8\pi}} J\ob{a + \frac{c^2}2,b}                                        \\
		         & = \frac{1}{8a} \ob{ 1 + \frac{b}{\sqrt{2a+b^2}} + \frac{c}{\sqrt{2a+c^2}} } \\&\quad+ \frac{1}{4\pi a} \ob{ \frac{b}{\sqrt{2a+b^2}}\arctan\frac{c}{\sqrt{2a+b^2}}+ \frac{c}{\sqrt{2a+c^2}}\arctan \frac{b}{\sqrt{2a+c^2}}}.
	\end{aligned}
\end{equation}

We can now use~\eqref{eq:mjoi21simp} and~\eqref{eq:Labcsol}
to express
\begin{equation}
	\begin{aligned}
		\E [M_\joi^2] = 2\Bigg[ & I_2\ob{\frac1{\si_1},\frac1{\si_2}}  + L\ob{2(1-\si_2^2), \frac 1{\si_1}, \frac{1-2\si_2^2}{\si_2}  }                                   + L\ob{ 2(1-\si_1^2), \frac 1{\si_2}, \frac{1-2\si_1^2}{\si_1} } \\
		                        & - L\ob{ 2(2 - \si_1^2 - \si_2^2), \frac{1 - 2\si_1^2}{\si_1}, \frac{1 - 2\si_2^2}{\si_2} }\Bigg].
	\end{aligned}
\end{equation}

\subsection{A remark on the asymptotics}\label{ssec:davydov}
Consider i.i.d.\ Gaussian-endpoint Brownian bridges \( X_1, X_2, \ldots, X_n \) with uniform parameter \( \sigma \in \cc{0,1} \).
We consider the asymptotics of their joint convex hull \( K_{\si, n} \), as well as its perimeter and area,
as the number of processes \( n \) tends to infinity. The key result of~\cite{davydov} establishes the convergence
\begin{equation}
	\frac 1{\sqrt{2 \log n}} K_n \to K
\end{equation}
in the Hausdorff sense on the metric space of
compact convex sets in \( \R^2 \). Here \( K_n \) denotes the
joint convex hull of processes \( Y_1, Y_2, \ldots, Y_n \),
which are i.i.d.\ copies of a general centered a.s.\ bounded Gaussian process.
This is a very broad class of processes that includes our Gaussian-endpoint Brownian bridges.
Meanwhile,
\( K = \hull\vit{E_t \st t \in \cc{0,1}} \) is the limiting shape,
\( E_t \) denoting the concentration ellipsoid of \( Y_n(t) \).

Fix  \( \si \in \cc{0,1} \).
Due to isotropy, each concentration ellipsoid is a disk centered at the origin, as is the limiting object
$
	K_\sigma = r_\si B(0,1)
$,
where \( B(0,1) \) is the unit disk in \( \R^2 \) and
\begin{equation}
	r_\si = \max_{0 \le t \le 1} t - (1-\sigma^2) t^2.
\end{equation}
By an elementary computation it follows that
\begin{equation}\label{eq:r-sigma-final}
	r_\sigma
	=
	\begin{cases}
		\dfrac{1}{2\sqrt{1-\sigma^2}},
		 & 0\le \sigma\le \dfrac1{\sqrt2},
		\\ 
		\sigma,
		 & \dfrac1{\sqrt2}\le \sigma\le1.
	\end{cases}
\end{equation}
Additionally,~\cite[Theorem 2]{davydov} establishes the convergence
of the expectations of functionals towards the said functional applied to the limiting object.
More precisely, the functionals have to be nonnegative, continuous, increasing and \( p \)-homogeneous. 
Both the perimeter and area are examples of such functionals (\( p=1,2 \) respectively).
To summarize this subsection,
\begin{equation}
	\frac 1{\sqrt{2 \log n}} K_{\sigma, n} \to r_\si B(0,1), \qquad
	\frac 1{\sqrt{2 \log n}}\E [\per K_{\si, n}] \to 2r_\si\pi, \qquad
	\frac 1{2 \log n}\E [\area K_{\si ,n}] \to r_\si^2 \pi,
\end{equation}
as \( n \ub \) and the first convergence being in the Hausdorff sense.

\appendix
\section{}
In this section, we prove two technical auxiliary results used in the paper.

\begin{lemma}\label{lem:a2}
	For every $\sigma \in [0,1]$, one has
	\begin{equation}\label{eq:a2}
		\int_0^1 t(1-t)\arctan\ob{\sigma\sqrt{\frac{1-t}{t}}}\D t
		=
		\frac{\pi\sigma(4\sigma^2+9\sigma+3)}{48(1+\sigma)^3}.
	\end{equation}
\end{lemma}

\begin{proof}
	Denote the integral in the statement of the lemma by $I(\sigma)$. Using the substitution \( x = \sqrt{(1-t)/t} \) we conclude that
	\begin{align}
		I(\sigma) & =
		\int_0^1 t(1-t)\arctan\ob{\sigma\sqrt{\frac{1-t}{t}}}\D t \\
		          & = 2\int_0^\infty
		\frac{x^3}{(1+x^2)^4}\arctan(\sigma x)\D x.
	\end{align}
	Differentiating with respect to $\sigma$ gives
	\begin{equation}
		I'(\sigma)=
		2\int_0^\infty
		\frac{x^4}{(1+x^2)^4(1+\sigma^2x^2)}\D x.
	\end{equation}
	To continue, rewrite the numerator as
	\begin{equation}
		x^4 = \ob{(1+x^2) - 1}^2 = (1+x^2)^2 - 2(1+x^2) + 1,
	\end{equation}
	so that
	\begin{equation}\label{eq:Ideri-split}
		I'(\sigma) = 2(J_2 - 2J_3 + J_4),
	\end{equation}
	where
	\begin{equation}
		J_n = \int_0^\infty \frac{\D x}{(1+x^2)^n(1+\sigma^2x^2)}.
	\end{equation}
	To find $J_n$, we establish a recursive formula. Starting with $n=1$,
	\begin{equation}
		\frac{1}{(1+x^2)(1+\sigma^2x^2)} = \frac{1}{1-\sigma^2} \ob{ \frac{1}{1+x^2} - \frac{\sigma^2}{1+\sigma^2x^2} },
	\end{equation}
	we divide by $(1+x^2)^{n-1}$ and integrate to obtain
	\begin{equation}\label{eq:Knrek}
		J_n = \frac{1}{1-\sigma^2} (L_n - \sigma^2 J_{n-1}), \qquad n \ge 2,
	\end{equation}
	where, substituting $x = \tan y$,
	\begin{align}
		L_n & = \int_0^\infty \frac{\D x}{(1+x^2)^n}           = \int_0^{\pi/2} \cos^{2n-2}y \D y           = 2^{2n-3} \up B \ob{\frac{2n-1}2,\frac{2n-1}2}
	\end{align}
	is known as Wallis' integral, and the final equality follows from \cite[3.621.1]{gradshteyn_ryzhik}.

	The base for the recursion,
	\begin{equation}
		J_1 = \frac \pi {2 (1 + \si)},
	\end{equation}
	is derived directly.
	Iteratively applying~\eqref{eq:Knrek} yields:
	\begin{align}
		J_2 =
		\frac{\pi(1+2\sigma)}{4(1+\sigma)^2},\quad
		J_3 =
		\frac{\pi(3+9\sigma+8\sigma^2)}{16(1+\sigma)^3}, \quad
		J_4 =
		\frac{\pi(5+20\sigma+29\sigma^2+16\sigma^3)}{32(1+\sigma)^4}.
	\end{align}
	Returning to~\eqref{eq:Ideri-split} and simplifying,
	\begin{equation}
		I'(\sigma)
		=
		\frac{\pi(\sigma^2+4\sigma+1)}{16(1+\sigma)^4}.
	\end{equation}
	By continuity, this formula also holds at $\sigma=1$. Since $I(0)=0$,
	\begin{equation}
		I(\sigma)
		=
		\frac{\pi}{16}
		\int_0^\sigma
		\frac{u^2+4u+1}{(1+u)^4}\D u.
	\end{equation}
	The derivation of~\eqref{eq:a2} is finished with the standard technique of partial fraction
	decomposition.
\end{proof}

\begin{lemma}\label{lma:uniqueargmax}
	For all \( \si \in \cc{0,1} \), the projected process \( \gproc x \si t = B_t \nderi 1 + \si t Y \) (see Subsection~\ref{ssec:projproc})
	a.s.\ attains its maximum at a unique point.
\end{lemma}

\begin{proof}
	As noted in Subsection~\ref{ssec:projproc}, conditionally on \( Y=y \) the process \( \gprocs x \si \)
	is a one-dimensional $0$-to-$y$ Brownian bridge on \( \cc{0,1} \). Fix  \( y\in \R \)
	and consider the conditional process denoted \( B_t^{0 \to y} \). We now use a standard technique---suppose
	\( B_t^{0 \to y} \) attains its maximum at two distinct times. Then for some rational \( q \in \oo{0,1} \),
	\begin{equation}\label{eq:equalmax}
		\max_{0 \le t \le q} B_t^{0 \to y} = \max_{q \le t \le1} B_t^{0 \to y}.
	\end{equation}
	The two pieces of the bridge on \( \cc{0,q} \) and \( \cc{q,1} \) are independent conditionally on the value \( B_q^{0 \to y} \).
	Moreover these are independent Brownian bridges with general endpoints. Their maxima are also conditionally independent
	and conditionally absolutely continuous (cf.\ Lemma~\ref{lma:br0ymax}). The probability of the event in~\eqref{eq:equalmax}
	is therefore \( 0 \). Taking a countable union over rational \( q \), we prove that \( B_t^{0 \to y} \) a.s.\ has
	a unique maximizing time. Integrating over \( Y=y \), the same follows for the process \( \gprocs  x \si \).
\end{proof}

\section*{Disclosure on the use of LLM technologies}
\noindent The authors disclose that they used ChatGPT and Gemini for aid in overviews of literature and known results, symbolic calculations and proofreading. The authors have verified all claims and calculations and have composed the manuscript independently. They take full responsibility for the correctness and accuracy of the paper's contents.

\section*{Acknowledgments}
\noindent NS and S\v{S} were supported by  the European Union – NextGenerationEU through the National Recovery and Resilience Plan 2021-2026 Institutional grant of University of Zagreb Faculty of Science (IK IA 1.1.3. Impact4Math), and Institutional grant of University of Zagreb Faculty of Electrical Engineering and Computing (VALOR). NS, S\v{S} and L\v{S} received support from grant DIGIT.2.1.02.016 funded by the Digital, Innovation, and Green Technology Project – DIGIT Project (IBRD Loan No. 9558‑HR).
NS, S\v{S} and L\v{S} were supported by  Croatian Science Foundation grant no.~2277.

\appendix

\printbibliography

\end{document}